\documentclass[10.5pt]{amsart}
\usepackage[english]{babel}
\usepackage{amsmath}
\usepackage{amssymb}
\usepackage{amsthm}
\usepackage{libertine}
\usepackage[all]{xy}
\usepackage{datetime}
 \newtheorem{thm}{Theorem}[section]
 \newtheorem{cor}[thm]{Corollary}
 \newtheorem{lem}[thm]{Lemma}
 \newtheorem{prop}[thm]{Proposition}
  
 \theoremstyle{definition}
 \newtheorem{defn}[thm]{Definition}
 \theoremstyle{remark}
 \newtheorem{rem}[thm]{Remark}
 \newtheorem*{ex}{Example}
 \numberwithin{equation}{section}
\newcommand{\Hq}{\mathbb H}
\newcommand{\Sq}{\mathbb S}
\newcommand{\N}{\mathbb{N}}

\newcommand{\R}{\mathbb{R}}      

\newcommand{\C}{\mathbb{C}}

\title[On the global operator and Fueter mapping theorem for slice polyanalytic functions]{On the global operator and Fueter mapping theorem for slice polyanalytic functions}

\author[D. Alpay]{Daniel Alpay}
\address{(DA) Schmid College of Science and Technology, Chapman University, Orange 92866, CA, US}
\email{alpay@chapman.edu}
\author[K. Diki]{Kamal Diki}
\address{(KD) Politecnico di
Milano\\Dipartimento di Matematica\\Via E. Bonardi, 9\\20133 Milano,
Italy}
\email{kamal.diki@polimi.it}
\author[I. Sabadini]{Irene Sabadini}
\address{(IS) Politecnico di
Milano\\Dipartimento di Matematica\\Via E. Bonardi, 9\\20133 Milano\\Italy}
\email{irene.sabadini@polimi.it}

\begin{document}
\maketitle
\begin{abstract}
In this paper, we prove that slice polyanalytic functions on quaternions can be considered as solutions of a power of some special global operator with nonconstant coefficients as it happens in the case of slice hyperholomorphic functions. We investigate also an extension version of the Fueter mapping theorem in this polyanalytic setting. In particular, we show that under axially symmetric conditions it is always possible to construct Fueter regular and poly-Fueter regular functions through slice polyanalytic ones using what we call the poly-Fueter mappings. We study also some integral representations of these results on the quaternionic unit ball.
 \end{abstract}

\noindent AMS Classification: Primary 30G35, 32A25, 30E20.

\noindent {\em Key words}: Quaternions, Slice poly-analytic functions, poly-Fueter regular functions, Poly-Fueter mapping.

\section{Introduction}
This paper proposes a bridge between two theories: the one of slice polyanalytic functions and the one of poly-Fueter regular functions. To understand the framework, we recall that in classical complex analysis, $n$-analytic or polyanalytic functions are null-solutions of the $n$-power of the Cauchy-Riemann operator.
In the quaternionic setting or, more in general, in the Clifford algebra setting, one can extend this notion by considering functions in the kernel of a generalized Cauchy-Riemann operator (thus obtaining the so-called regular or monogenic functions, see \cite{CSSS2004,GHS}) or of its $n$-power (thus obtaining poly-regular functions or poly-monogenic functions, see \cite{PSS2016, KahlerKuQian2017}).

 This was the first approach to extend  holomorphic functions, and then polyanalytic functions, to a higher dimensional setting. Although this theory is very important and well developed, it does not include the class of elementary functions and power series. In 2006, a new approach to quaternionic analyticity was proposed in the literature \cite{GS2006, GS2007}, namely the so-called slice hyperholomorphic (or slice regular) function theory and it is now widely studied, see the books \cite{ACS_book,CSS,CSS_book, GentiliSS}. It founds some interesting applications in different areas of mathematics and physics, see \cite{FJBOOK,CGKBOOK}. It is interesting to note that the class of slice hyperholomorphic functions is related with the class of functions considered by Fueter to construct regular functions and thus there is a bridge between them, specifically the so-called Fueter mapping, in fact by applying the Laplacian to a slice hyperhomolorphic function one obtains a regular function, i.e. a function in the kernel of the Cauchy-Fueter operator, see for example  \cite{ColomboSabadiniSommen2010}. Also the theory of polyanalytic functions can be extended to the slice setting by considering a suitable definition, as we did in \cite{ADS2019}. Thus it is a natural question to ask whether there is an analog of the Fueter map in this more general setting. The answer is positive and it is one of the main results of this paper: we show that by applying the Laplacian composed with the $(n-1)$ power of the global operator $V=2\overline{\vartheta}$ (where $\overline{\vartheta}$ is the operator introduced in \cite{P2019}) to any slice polyanalytic function of order $n$ we obtain a Cauchy-Fueter regular function. A second approach to extend the Fueter mapping to the polyanalytic setting consists to apply the standard Fueter mapping on each component associated to the poly-decomposition. This constrution allows to generate poly-Fueter regular functions starting from slice polyanalytic ones of the same order.

To put our work in perspective, we recall that classical polyanalytic functions are important not only from the theoretical point of view, see the classical book \cite{Balk1991}, but also in the theory of signals since they allow to encode $n$ independent analytic functions into a single polyanalytic one using a special decomposition. This idea is similar to the problem of multiplexing signals. This is related to the construction of the polyanalytic Segal-Bargmann transform mapping $L^2(\R)$ onto the poly-Fock space, see \cite{AF2014}. In quantum physics these functions are relevant for the study of the Landau levels associated to Schr\"odinger operator, see \cite{AF2014,AIM1997}. Polyanalytic functions were used also in \cite{A2010} to study sampling and interpolation problems on Fock spaces using time frequency analysis techniques such as short-time Fourier transform (STFT) or Gabor transforms. This allows to extend Bargmann theory to the polyanalytic setting using Gabor analysis. The theory of signals is widely studied also with hypercomplex methods and for a list of references the reader may consult \cite{ck} and the references therein.

As we said, Fueter regular and slice hyperholomorphic functions are related  by the famous Fueter mapping theorem. This result has some important consequences and allows to define the $\mathcal{F}$-functional calculus for quaternionic operators with commuting components. Recently, new several results  for polyanalytic functions were proven in the slice hyperholomorphic context over the quaternions, see \cite{ADS2019}, and the counterparts of the Bergman and Fock spaces were also considered. In this paper we continue the investigations in this direction. In particular, we prove a new version of the well-known Fueter mapping theorem that will relate slice and Cauchy-Fueter polyanalytic functions on quaternions and present an integral form of this result.

The paper has the following structure: in Section 2 we set up basic notations and revise some preliminary results. Section 3 contains some results on the powers of the global operator $V$ and the main statements and proofs of the poly-Fueter mapping theorems.  In Section 4 we study an integral representation of these results based on the poly-Cauchy formula on the quaternionic unit ball. In Section 5, we rewrite our results in the slice polymonogenic case.

\section{ Preliminary results}
We revise different notions and results related to  Cauchy-Fueter and slice hyperholomorphic functions and  also the polyanalytic setting on quaternions. Different versions of the Fueter mapping theorem are also recalled.
The non-commutative field of quaternions is defined to be
$$\Hq=\lbrace{q=x_0+x_1i+x_2j+x_3k : \ x_0,x_1,x_2,x_3\in\R}\rbrace$$ where the imaginary units satisfy the multiplication rules $$i^2=j^2=k^2=-1\quad \text{and}\quad ij=-ji=k, jk=-kj=i, ki=-ik=j.$$
On $\Hq$ the conjugate and the modulus of $q$ are defined respectively by
$$\overline{q}=x_0-\vec{q}\,,  \quad \vec{q}\,=x_1i+x_2j+x_3k$$
and $$\vert{q}\vert=\sqrt{q\overline{q}}=\sqrt{x_0^2+x_1^2+x_2^2+x_3^2}.$$

Sometimes, we use also the notations $e_0=1, e_1=i, e_2=j$ and $e_3=k$ for the standard imaginary units.
We note that the quaternionic conjugation satisfy the property $\overline{ pq }= \overline{q}\, \overline{p}$ for any $p,q\in \Hq$.
Moreover, the unit sphere $$\lbrace{q=x_1i+x_2j+x_3k : \text{ } x_1^2+x_2^2+x_3^2=1}\rbrace$$ coincides with the set of all  imaginary units given by $$\mathbb{S}=\lbrace{q\in{\Hq} : q^2=-1}\rbrace.$$
  Any quaternion $q\in \Hq\setminus \R$ can be written in a unique way as $q=x+I y$ for some real numbers $x$ and $y>0$, and imaginary unit $I\in \mathbb{S}$. For every given $I\in{\mathbb{S}}$, we define $\C_I = \mathbb{R}+\mathbb{R}I.$
It is isomorphic to the complex plane $\C$ so that it can be considered as a complex plane in $\Hq$ passing through $0$, $1$ and $I$. Their union is the whole space of quaternions $$\Hq=\underset{I\in{\mathbb{S}}}{\cup}\C_I =\underset{I\in{\mathbb{S}}}{\cup}(\mathbb{R}+\mathbb{R}I).$$
Let $\mathbb{B}$ denotes the quaternionic unit ball and $\mathbb{B}_I$ its intersection with the complex plane $\C_I$ for a given $I\in\mathbb{S}$. Then, we recall
\begin{defn} Let $U\subset \Hq$ be an open set and let $f:U\longrightarrow \Hq$ be a function of class $\mathcal{C}^1$. We say that $f$ is (left) Fueter regular or regular for short on $U$ if $$\mathcal{D} f(q):=\displaystyle\left(\frac{\partial}{\partial x_0}+i\frac{\partial}{\partial x_1}+j\frac{\partial}{\partial x_2}+k\frac{\partial}{\partial x_3}\right)f(q)=0, \forall q\in U.$$
\\
The right quaternionic vector space of Fueter regular functions will be denoted by $\mathcal{FR}(U)$.
\end{defn}

\begin{defn}
Let $f: \Omega \longrightarrow \Hq$ be a  $\mathcal{C}^1$ function on a given domain $\Omega\subset \Hq$. Then, $f$ is said to be (left) slice hyperholomorphic function if, for every $I\in \Sq$, the restriction $f_I$ to $\C_{I}=\R+I\R$, with variable $q=x+Iy$, is holomorphic on $\Omega_I := \Omega \cap \C_I$, that is it has continuous partial derivatives with respect to $x$ and $y$ and the function
$\overline{\partial_I} f : \Omega_I \longrightarrow \Hq$ defined by
$$
\overline{\partial_I} f(x+Iy):=
\dfrac{1}{2}\left(\frac{\partial }{\partial x}+I\frac{\partial }{\partial y}\right)f_I(x+yI)
$$
vanishes identically on $\Omega_I$. 
\end{defn}
Later we shall introduce another definition of slice regularity that is a special case of the definition in the poly analytic case with $n=1$.
The right quaternion vector space of slice hyperholomorphic functions is endowed with the natural topology of uniform convergence on compact sets. The characterization of such functions on a ball centered at the origin is given by

\begin{thm}[Series expansion \cite{GS2007}]
An $\Hq$-valued function $f$ is slice hyperholomorphic on $B(0,R)$ if and only if it has a series expansion of the form:
$$f(q)=\sum_{n=0}^{+\infty} q^na_n$$
converging on $B(0,R)=\{q\in\Hq;\mid q\mid<R\}$.
\end{thm}
Another approach to define slice hyperholomorphic functions is to consider them as solutions of a special global operator with non constant coefficients that was introduced and studied in \cite{CGCS2013, CS2014, P2019}. This leads to the following definition
\begin{defn}
Let $\Omega$ be an open set in $\Hq$ and $f:\Omega\longrightarrow \Hq$ a function of class  $\mathcal{C}^1$. We define the global operator $G(f)$ by
$$\displaystyle G(f)(q):=|\vec{q}\,|^2\partial_{x_0}f(q)+\vec{q}\,\sum_{l=1}^3x_l\partial_{x_l}f(q),$$
for any $q=x_0+\vec{q}\,\in\Omega$.
\end{defn}
It was proved in \cite{CGCS2013} that any slice hyperholomorphic function belongs to $\ker(G)$ on axially symmetric slice domains. We recall briefly this notion:

\begin{defn}
A domain $\Omega\subset \Hq$ is said to be a slice domain (or just $s$-domain) if  $\Omega\cap{\mathbb{R}}$ is nonempty and for all $I\in{\mathbb{S}}$, the set $\Omega_I:=\Omega\cap{\C_I}$ is a domain of the complex plane $\C_I$.
If moreover, for every $q=x+Iy\in{\Omega}$, the whole sphere $x+y\mathbb{S}:=\lbrace{x+Jy; \, J\in{\mathbb{S}}}\rbrace$
is contained in $\Omega$, we say that  $\Omega$ is an axially symmetric slice domain.
\end{defn}
 Other interesting properties of the global operator $G$ were studied in \cite{CGCS2012}. We recall some of them that will be helpful for our purposes:
\begin{prop}\label{GPRO}
Let $\Omega$ be an open set in $\Hq$ and $f,g:\Omega\longrightarrow \Hq$ two functions of class $\mathcal{C}^1$. Then, for $q=x_0+\vec{q}\,\in\Omega$ we have
\begin{enumerate}
\item $\displaystyle G(fg)=G(f)g+fG(g)+(\vec{q}\,f-f\vec{q}\,)\sum_{l=1}^3x_l\partial_{x_l}g$. \newline \text{In particular, it holds:}

\item $G(f\lambda+g)=G(f)\lambda+G(g),  \forall \lambda \in \Hq$.
\item $G(x_0f)=|\vec{q}\,|^2f+x_0G(f)$ and $G(\vec{q}\,f)=-|\vec{q}\,|^2f+\vec{q}\,G(f).$
\item $G(q^kf)=q^kG(f), \forall k\in\N.$

\end{enumerate}
\end{prop}

 We recall below the variations of the Fueter mapping theorem that we will use later in this paper and refer the reader to \cite{ColomboSabadiniSommen2010,Qian2014} for several extensions.
\begin{thm}[Fueter mapping theorem \cite{ColomboSabadiniSommen2010}]\label{FMthm}
Let $U$ be an axially symmetric set in $\mathbb H$ and let $f:U\subset \Hq\longrightarrow \Hq$ be a slice hyperholomorphic function of the form $f(x+yI)=\alpha(x,y)+I\beta(x,y)$,where $\alpha(x,y)$ and $\beta(x,y)$ are quaternionic-valued functions such that $\alpha(x,-y)=\alpha(x,y)$, $\beta(x,-y)=-\beta(x,y)$ and satisfying the Cauchy-Riemann system. Then, the function $$\overset{\sim}f(x_0+\vec{q}\,)=\displaystyle \Delta\left(\alpha(x_0,\vert{\vec{q}\,}\vert)+\frac{\vec{q}\,}{{|\vec{q}\,|}}\beta(x_0,\vert{\vec{q}\,}\vert)\right)$$
extends to a Fueter regular function on the whole $U$.
\end{thm}
\begin{rem}\label{FMrem}
If $U$ is an axially symmetric slice domain in $\mathbb{H}$, then every slice hyperholomorphic function $f:U\subset\Hq\longrightarrow \Hq$ is of the form $f(x+Iy)=\alpha(x,y)+I\beta(x,y)$, where $\alpha$ and $\beta$ have the properties mentioned in the preceding statement. This is an immediate consequence of the Representation formula observed in  Lemma 2.2 in \cite{CS}.
\end{rem}
 A function $f(x+yI)=\alpha(x,y)+I\beta(x,y)$, where $\alpha, \beta$ are $\mathbb H$ (or $\mathbb R_n$)-valued, $\alpha(x,-y)=\alpha(x,y)$, $\beta(x,-y)=-\beta(x,y)$ is called a slice function.
\begin{rem}
We denote by $\mathcal{SR}(U)$ the space of slice regular functions which are slice functions. Below, we can consider the Fueter mapping defined by $$\tau:\mathcal{SR}(U)\rightarrow\mathcal{FR}(U), \text{ } f\longmapsto \tau(f)=\Delta(f).$$
\end{rem}

\begin{thm}[\cite{ColomboSabadiniSommen2010}] \label{CFK}
Given a quaternion $s\in\Hq$, we define
$$[s]=\lbrace p\in\Hq: \textbf{ } p=Re(s)+I|\vec{s}\,|, I\in\mathbb{S} \rbrace .$$
Let $S^{-1}(s,q)$ be the Cauchy kernel defined by:
$$S^{-1}(s,q)=(s-\overline{q})(s^2-2Re(q)s+|q|^2)^{-1}, \textbf{ } q\notin [s].$$

Then the function
$$\displaystyle \mathcal{F}(s,q):=\Delta S^{-1}(s,q)=-4(s-\overline{q})(s^2-2Re(q)s+|q|^2)^{-2},$$
is a Cauchy-Fueter regular function in the variable $q$, and it is right slice regular in the variable $s$ for $q\notin [s]$.
\end{thm}
\begin{thm}[The Fueter mapping theorem in integral form \cite{ColomboSabadiniSommen2010}]
Let $W\subset\mathbb{H}$ be an axially symmetric open set and let $f$ be slice hyperholomorphic in $W$. Let $U$ be a bounded axially symmetric open set such that $\overline{U}\subset W$.  Suppose that the boundary of $U_I=U\cap\C_I$ consists of finite number of rectifiable Jordan curves for any $I\in\mathbb{S}$. Then, if $q\in U$, the Cauchy–Fueter regular function given by

 $$\tau(f)(q)=\Delta f(q)$$
has the integral representation
$$\tau(f)(q)=\displaystyle\frac{1}{2\pi} \int_{\partial U_I} \Delta S^{-1}(s,q)ds_If(s), \textbf{  } ds_I=ds/I,$$
and the integral does not depend on $U$ nor on the imaginary unit $I\in\mathbb{S}.$
\end{thm}

We will need also these useful results in our computations
\begin{prop}[\cite{Begehr1999}] \label{DiracAct}
For all $n\geq 2$, we have
$$\mathcal{D} [q^n]=-2\displaystyle\sum_{k=1}^{n}q^{n-k}\overline{q}^{k-1}.$$
\end{prop}
\begin{prop}[\cite{DKS2019}] \label{FueterAct}
For all $n\geq 2$, we have
$$\tau [q^n]=-4\displaystyle\sum_{k=1}^{n-1}(n-k)q^{n-k-1}\overline{q}^{k-1}.$$
\end{prop}
%\subsection{Slice polyanalytic and Fueter polyregular functions}
In \cite{ADS2019} the theory of slice hyperholomorphic functions on quaternions is extended to higher order by considering:
\begin{defn}
Let $\Omega$ be an axially symmetric open set in $\Hq$ and let $f:\Omega\longrightarrow \Hq$ a slice function of class $\mathcal{C}^n$. For each $I\in\Sq$, let $\Omega_I=\Omega\cap\C_I$ and let $f_I=f_{\vert_{\Omega_I}}$ be the restriction of $f$ to $\Omega_I$. The restriction $f_I$ is called (left) polyanalytic of order $n$ if it satisfies on $\Omega_I$ the equation $$
\overline{\partial_I}^n f(x+Iy):=
\frac{1}{2^n}\left(\frac{\partial }{\partial x}+I\frac{\partial }{\partial y}\right)^nf_I(x+Iy)=0.
$$
The function $f$ is called left slice polyanalytic of order $n$, if for all $I\in\Sq$, $f_I$ is left polyanalytic of order $n$ on $\Omega_I$. The right quaternionic vector space of slice polyanalytic functions of order $n$ will be denoted by $\mathcal{SP}_n(U)$.
\end{defn}
Note that slice regular functions are a special case of the definition of slice polyanalytic functions with $n=1$. The right slice polyanalytic functions can be defined in a similar way just by taking the powers  of the Cauchy-Riemann operator with imaginary unit on the right. Several results of these functions were studied and extended. In particular, we recall some properties that we need for our computations in the next sections.
\begin{prop}[Splitting Lemma]\label{split} Let $f$ be a slice polyanalytic function of order $n$ on an axially symmetric domain $\Omega\subseteq\Hq$. Then, for any imaginary units $I$ and $J$ with $I\perp J$ there exist $F,G:\Omega_{I}\longrightarrow{\C_I}$ polyanalytic functions of order $n$ such that for all $z=x+Iy\in\Omega_I$, we have
$$f_I(z)=F(z)+G(z)J.$$
 \end{prop}
 We will be interested also by the following decomposition
\begin{prop}[Poly-decomposition]\label{carac3}
A function $f:\Omega\longrightarrow \Hq$ defined on an axially symmetric slice domain is slice polyanalytic of order $n$ if and only there exist $f_0,...,f_{n-1}$ some unique slice hyperholomorphic functions on $\Omega$ such that we have the following decomposition: $$
f(q):=\displaystyle\sum_{k=0}^{n-1}\overline{q}^kf_k(q); \textbf{ }\forall q\in\Omega.
$$
\end{prop}
Finally, we consider the poly-Fueter regular functions that can be found for example in \cite{KahlerKuQian2017} for Clifford valued functions.
\begin{defn} Let $U\subset \Hq$ be an open set and let $f:U\longrightarrow \Hq$ be a function of class $\mathcal{C}^n$. We say that $f$ is (left) poly-Fueter regular or poly-regular for short of order $n\geq 1$ on $U$ if $$\mathcal{D}^n f(q):=\displaystyle\left(\frac{\partial}{\partial x_0}+i\frac{\partial}{\partial x_1}+j\frac{\partial}{\partial x_2}+k\frac{\partial}{\partial x_3}\right)^nf(q)=0, \forall q\in U.$$
\\
The right quaternionic vector space of poly-Fueter regular functions will be denoted by $\mathcal{FR}_n(U)$.
\end{defn}
The proof of the next result was communicated to us by Dan Volok, and appears earlier in section 6 and 7 of  \cite{B1976}, see also \cite{DB1978} for the Clifford monogenic setting. We recall it for completeness
\begin{prop}\label{Poly-Fueter-Dec}
A function $f$ is poly-Fueter regular of order $n$ if and only if it can be decomposed in terms of some unique Fueter regular functions $\phi_0,...,\phi_{n-1}$ such that we have $$f(q)=\displaystyle \sum_{k=0}^{n-1}x_0^k\phi_k(q).$$
\end{prop}

\section{The global operator and poly-Fueter mapping theorem }
In this section, we show that slice polyanalytic functions of some order $n$ are solutions of the n-th power of a certain global operator $V$. A new extension of the Fueter mapping theorem involving slice polyanalytic functions on quaternions will be proved also.

In \cite{P2019}, the authors considered a modified version of the operator $G$ which is defined by $$\displaystyle V(f)(q):=\partial_{x_0}f(q)+\frac{\vec{q}\,}{|\vec{q}\,|^2}\sum_{l=1}^3x_l \partial_{x_l}f(q), \forall q\in\Omega\setminus \R.$$
\begin{rem}\label{GVR}
For suitable domains, we note that the operators $G$ and $V$ are related by the formula
$$V(f)(q)=\frac{1}{|\vec{q}\,|^2}G(f)(q), \qquad \forall q\in \Omega\setminus \R.$$
In what follows, if $V(f)$ admits a (unique) continuous extension on the whole $\Omega$, then we implicitly assume that $V(f)$ denotes such an extension. Given any $n\geq 2,$ inductively we will say that $V^n(f)$ is a function on $\Omega$ if $V^{n-1}(f)$ is of class $\mathcal{C}^1$ on $\Omega\setminus \R$ and $V^n(f):=V(V^{n-1}(f))$ admits a continuous extension on $\Omega$.
\end{rem}
First, we prove some preliminary results on the global operators $G$ and $V$ that  are needed in the sequel.
\begin{lem} \label{Gqbar}
Let $\Omega$ be an open set in $\Hq$ and $\psi:\Omega\longrightarrow \Hq$ a function of class $\mathcal{C}^1$. Then, we have
$$\displaystyle G(\overline{q}\psi)(q)=\overline{q}G(\psi)(q)+2|\vec{q}\,|^2\psi(q), \textbf{ }\forall q=x_0+\vec{q}\,\in\Omega.$$
\end{lem}
\begin{proof}
Let $\psi$ be a $\mathcal{C}^1$ function on $\Omega$, we apply the definition of $G$ and Leibniz rule with respect to the partial derivatives and we get
\[ \begin{split}
 \displaystyle G(\overline{q}\psi)(q) &:= |\vec{q}\,|^2\partial_{x_0}(\overline{q}\psi)(q)+\vec{q}\,\sum_{l=1}^3x_l\partial_{x_l}(\overline{q}\psi)(q) \\
 &=|\vec{q}\,|^2\overline{q}\partial_{x_0}\psi(q)+|\vec{q}\,|^2\psi(q)+\vec{q}\, \textbf{ }\overline{q}\sum_{l=1}^3x_l\partial_{x_l}\psi(q) -\vec{q}\,\sum_{l=1}^3x_le_l\psi(q).\\
\end{split}
\]
However, we know that
$$\displaystyle \vec{q}\,=\sum_{l=1}^3x_le_l, \quad \vec{q}\,\textbf{ }\overline{q}=\overline{q}\vec{q}\, \quad \text{ and } \quad \vec{q}\,^2=-|\vec{q}\,|^2.$$
Thus, for any $q\in\Omega$ we have
\[ \begin{split}
 \displaystyle G(\overline{q}\psi)(q) &= \overline{q}\left(|\vec{q}\,|^2\partial_{x_0}\psi(q)+\vec{q}\,\sum_{l=1}^3x_l\partial_{x_l}\psi(q)\right)+2|\vec{q}\,|^2\psi(q)  \\
 &=\overline{q}G(\psi)(q)+2|\vec{q}\,|^2\psi(q) .\\
\end{split}
\]
This ends the proof.
\end{proof}
\begin{cor}\label{GVC}
Let $\Omega\subseteq\mathbb H$ be a domain and
$f:\Omega\longrightarrow \Hq$ be a slice hyperholomorphic function. Then,  we have
$$\displaystyle G(\overline{q}f)(q)=2|\vec{q}\,|^2f(q), \textbf{ }\forall q\in\Omega$$
and
\begin{equation}\label{Verre}
\displaystyle V(\overline{q}f)(q)=2f(q), \textbf{ }\forall q\in\Omega.
\end{equation}
\end{cor}
\begin{proof}
The fact that $f$ is slice hyperholomorphic on $\Omega$ implies that $$G(f)(q)=0, \textbf{  } \forall q\in \Omega.$$
Hence, a direct application of Lemma \ref{Gqbar} gives \eqref{Verre} on $\Omega\setminus\mathbb{R}$. However, since the right hand side of \eqref{Verre} extends the left hand side to all of $\Omega$ as a slice hyperholomorphic function, then \eqref{Verre} holds on $\Omega$.
\end{proof}
\begin{ex} To provide an example, let us consider the particular case $f\in\mathcal{SP}_2(\Hq)$. Then, we have \begin{enumerate}
\item $\displaystyle V^2(f)(q)=0, \textit{ }\forall q\in\Hq.$
\item $\Delta V(f)$ is Cauchy-Fueter regular on $\Hq$.
\item $\overline{\mathcal{D}}V(f)$ is poly-Fueter regular of order $2$, where $\overline{\mathcal{D}}$ is the conjugate of the Cauchy-Fueter operator.
\end{enumerate}
To see that (1) holds, we use the poly-decomposition that asserts the existence of some unique functions $f_0,f_1\in \mathcal{SR}(\Hq)$ such that $$f(q)=f_0(q)+\overline{q}f_1(q), \forall q\in\Hq.$$
An application of corollary \ref{GVC} combined with the fact that slice hyperholomorphic functions belong to $\ker(V)$ show that (1) holds. The other two assertions follows similarly.
\end{ex}
\begin{prop} \label{GVAct}
Let $\Omega$ be an open set in $\Hq$ and $f:\Omega\longrightarrow \Hq$ a slice hyperholomorphic function. Let $n\geq 2$ and $1 \leq k\leq n-1$, then we have
\begin{enumerate}
\item $\displaystyle G(\overline{q}^kf)(q)=2k|\vec{q}\,|^2\overline{q}^{k-1}f(q), \textbf{ }\forall q\in\Omega.$
\item $\displaystyle V(\overline{q}^kf)(q)=2k\overline{q}^{k-1}f(q), \textbf{ }\forall q\in\Omega.$
\end{enumerate}
\end{prop}
\begin{proof}
Let $f\in\mathcal{SR}(\Omega)$ and $n\geq 2$. We reason by induction with respect to $n$.

(1) First, we note that the result holds for $n=2$ as a consequence of Corollary \ref{GVC}. Now, let $n\geq 2$ be such that we have
$$\displaystyle G(\overline{q}^kf)(q)=2k|\vec{q}\,|^2\overline{q}^{k-1}f(q), \textbf{ }\forall q\in\Omega, \forall 1 \leq k\leq n-1.$$
In order, to prove that the result holds for $n+1$, we only have to show that
\begin{equation}
\displaystyle G(\overline{q}^nf)(q)=2n|\vec{q}\,|^2\overline{q}^{n-1}f(q), \forall q\in \Omega.
\end{equation}
Indeed, we apply Lemma \ref{Gqbar} and obtain
\[ \begin{split}
 \displaystyle G(\overline{q}^n\psi)(q) &=G(\overline{q} \textbf{ } \overline{q}^{n-1}f)(q)  \\
 &=\overline{q}G(\overline{q}^{n-1}f)(q)+2|\vec{q}\,|^2\overline{q}^{n-1} f(q),  \forall q\in\Omega .\\
\end{split}
\]
However, by induction hypothesis we know that
$$\displaystyle G(\overline{q}^{n-1}f)(q)=2(n-1)|\vec{q}\,|^2\overline{q}^{n-2}f(q), \forall q\in \Omega.$$
Therefore, we get
\[ \begin{split}
 \displaystyle G(\overline{q}^nf)(q) &=2(n-1)|\vec{q}\,|^2\overline{q}^{n-1}f(q)+2|\vec{q}\,|^2\overline{q}^{n-1}f(q) \textbf{ }   \\
 &=2n|\vec{q}\,|^2\overline{q}^{n-1}f(q),  \forall q\in\Omega .\\
\end{split}
\]
Hence, the result holds by induction and this completes the proof.

(2) We know by Remark \ref{GVR} that $$V(f)(q)=\frac{1}{|\vec{q}\,|^2}G(f)(q), \forall q\in \Omega\setminus \R.$$

Then, since $f$ is a slice hyperholomorphic function on $\Omega$, the right hand side extends the left hand side as polyanalytic function of order $k$ and so we get
$$\displaystyle V(\overline{q}^kf)(q)=2k\overline{q}^{k-1}f(q), \textbf{ }\forall q\in\Omega, \ \  1 \leq k\leq n-1 .$$

\end{proof}
\begin{prop}\label{VO}
Let $\Omega$ be an axially symmetric slice domain in $\Hq$ and $f:\Omega\longrightarrow \Hq$ a slice polyanalytic function of order $n\geq 1$. Then, $V(f)$ is a slice polyanalytic function of order $n-1$ on $\Omega.$
\end{prop}
\begin{proof}
We note that $\Omega$ is a slice domain. So, by poly-decomposition there exist some unique slice regular functions $\varphi_0,...,\varphi_{n-1}$ such that we can write $$f(q)=\displaystyle\sum_{k=0}^{n-1}\overline{q}^k\varphi_k(q), \forall q\in \Omega.$$
Thus, by Proposition \ref{GVAct} we know that for all $q\in\Omega$ we have
\[ \begin{split}
 \displaystyle V(f)(q) &=\sum_{k=1}^{n-1}V(\overline{q}^k\varphi_k)(q)   \\
 &=2\sum_{k=1}^{n-1}k\overline{q}^{k-1}\varphi_k(q)\\
 &=\sum_{h=0}^{n-2}\overline{q}^{h}\zeta_{h}(q), \\
\end{split}
\]
where we have set $\zeta_h(q)=2(h+1)\varphi_{h+1}$, $\forall 0 \leq h \leq n-2$ which are slice hyperholomorphic functions on the whole $\Omega$ by hypothesis. Hence, $V(f)$ extends as a slice polyanalytic function of order $n-1$ on $\Omega$.
\end{proof}
\begin{thm} \label{Vn}
Let $\Omega$ be an axially symmetric slice domain in $\Hq$ and $f:\Omega\longrightarrow \Hq$ a slice polyanalytic function of order $n\geq 1$. Then, $f$ belongs to $\ker(V^n)$, i.e:
$$\displaystyle V^n(f)(q)=0, \textit{ }\forall q\in\Omega.$$
\end{thm}
\begin{proof}
We apply Proposition \ref{VO} iteratively and obtain
 $$V(f)\in\mathcal{SP}_{n-1}(\Omega), V^2(f)\in\mathcal{SP}_{n-2}(\Omega),...,V^{n-1}(f)\in\mathcal{SP}_{1}(\Omega)=\mathcal{SR}(\Omega). $$
In particular, we deduce that $V^{n-1}(f)$ is a slice hyperholomorphic function on $\Omega$. Therefore, it belongs to the kernel of the global operator $V$ outside the real line. Hence, since $V^{n-1}(f)$ admits a continuous extension to the whole $\Omega$, by Theorem 2.4 in \cite{P2019}, we conclude that
$$\displaystyle V^n(f)(q)=V(V^{n-1})(f)(q)=0, \textit{ }\forall q\in\Omega.$$
This ends the proof.
\end{proof}
\begin{thm}[Poly-Fueter mapping theorem I]
Let $\Omega$ be an axially symmetric slice domain in $\Hq$ and $f:\Omega\longrightarrow \Hq$ a slice polyanalytic function of order $n\geq 1$. Then the function given by
$$\displaystyle \tau_n(f)(q)=\Delta \circ V^{n-1}(f)(q), \textit{ }\forall q\in\Omega$$
belongs to the kernel of the Cauchy-Fueter operator $\mathcal{D}$.
\end{thm}
\begin{proof}
Using the same argument used to prove Theorem \ref{Vn}, we deduce that $V^{n-1}(f)$ is a slice hyperholomorphic function on $\Omega$. Therefore, since $\Omega$ is an axially symmetric slice domain we can use Theorem \ref{FMthm} and Remark \ref{FMrem} to conclude that the function $\tau_n(f)$
is in the kernel of the Cauchy-Fueter operator $\mathcal{D}$ on $\Omega$, i.e.,

$$\displaystyle \mathcal{D} \circ \tau_n(f)(q)=\mathcal{D}\circ \Delta \circ V^{n-1}(f)(q)=0, \textit{ }\forall q\in\Omega.$$
\end{proof}
\begin{rem}
We note that the poly-Fueter mapping $$\tau_n:=\Delta \circ V^{n-1}$$ takes the space of slice polyanalytic functions of order $n\geq 1$ into the space of Cauchy-Fueter regular functions $\mathcal{FR}(\Omega)$.
\end{rem}
\begin{thm}\label{APP}
Let $\Omega$ be an axially symmetric slice domain of $\Hq$ and $f:\Omega\longrightarrow \Hq$ a slice hyperholomorphic function. Let $n\geq 1$ and consider the functions defined by
$$\Psi_{f}^k(q):=\overline{q}^kf(q), \textbf{  } \forall q\in \Omega, \forall 0 \leq k \leq n-1.$$

Then, the family $\lbrace{\Psi_{f}^k}\rbrace_{ 0 \leq k \leq n}$ forms an Appell system with respect to the operator $\frac{1}{2}V$, namely

$$\displaystyle \frac{1}{2}V(\Psi_{f}^0)=0 \text{ and } \frac{1}{2}V(\Psi_{f}^k)=k\Psi_{f}^{k-1},  \forall 1 \leq k \leq n-1.$$
\end{thm}
\begin{proof}
 The function $\Psi^0_f=f$ is slice hyperholomorphic on $\Omega$. So, $f$ belongs to the kernel of the global operator $V$ on $\Omega$. Thus, we have $\displaystyle \frac{1}{2}V(\Psi_{f}^0)=0$. On the other hand, we know by  Proposition \ref{GVAct} that

$$\displaystyle V(\overline{q}^kf)(q)=2k\overline{q}^{k-1}f(q), \textbf{ }\forall q\in\Omega, 1 \leq k\leq n-1 .$$

 Therefore, this combined with Theorem 2.4 in \cite{P2019} allows to see that for all $q\in\Omega$ and $1 \leq k\leq n-1 $ we have
\[ \begin{split}
 \displaystyle \frac{1}{2}V(\Psi_f^k)(q) &=\frac{1}{2}V(\overline{q}^kf)(q)   \\
 &=k\overline{q}^{k-1}f(q)\\
 &=k\Psi_f^{k-1}(q). \\
\end{split}
\]
 This ends the proof.
\end{proof}
\begin{cor}
 The sequence $\lbrace{\overline{q}^k}\rbrace_{k\geq 0}$ is an Appell system with respect to $\frac{1}{2}V$.
\end{cor}
\begin{proof}
If we take the constant function $f=1$, we immediately obtain the result.
\end{proof}
\begin{rem}
We note that for any slice hyperholomorphic function $f$ the family $\lbrace{\Psi_{f}^k}\rbrace_{ 0 \leq k \leq n}$ considered in Theorem \ref{APP} form also an Appell system with respect to the Cauchy-Riemann operator $\displaystyle \frac{1}{2}\overline{\partial_I}$ for all $I\in\mathbb{S}$.
\end{rem}

The next result allows to construct poly-Fueter regular functions starting from slice polyanalytic ones of the same order:
\begin{thm}[Poly-Fueter mapping theorem II]
Let $\Omega\subseteq\mathbb H$ be an axially symmetric slice domain and let $f:\Omega\longrightarrow \Hq$ a slice polyanalytic function of order $n\geq 1$. Assume that $f$ admits the decomposition

$$\displaystyle f(q)=\sum_{k=0}^{n-1}\overline{q}^kf_k(q), \forall q\in\Omega$$ where $f_0,...,f_{n-1}\in \mathcal{SR}(\Omega).$
Then, the function defined by
\begin{equation}
\displaystyle \mathcal{C}_n(f)(q)=\sum_{k=0}^{n-1}x_0^k\Delta(f_k)(q), \forall q\in \Omega
\end{equation}

is a poly-Fueter regular function of order $n$.
\end{thm}
\begin{proof}
We note by Theorem \ref{FMthm} and Remark \ref{FMrem} that the functions $\phi_k=\Delta(f_k)$ are all Cauchy-Fueter regular on $\Omega$ for any $0 \leq k\leq n-1$. Hence, thanks to Proposition \ref{Poly-Fueter-Dec} we conclude that $\mathcal{C}_n(f)$ is poly-Fueter regular of order $n$.

\end{proof}

Let $n\geq 1$, the two poly-Fueter mappings $\tau_n$ and $\mathcal{C}_n$ can be related to each other so that we have $$\displaystyle \tau_n:=(2\mathcal{D})^{n-1}\circ \mathcal{C}_n,$$
in other words the  diagram
 \[
\xymatrix{
\mathcal{SP}_n \ar[r]^{\tau_n} \ar[d]_{\mathcal{C}_n}& \mathcal{FR}\\
\mathcal{FR}_n\ar[ru]_{(2\mathcal{D})^{n-1}} }
\]
is commutative.

The proof of this fact is contained in the next result:
\begin{thm}\label{TauC}
Let $f:\Omega\longrightarrow \Hq$ be a slice polyanalytic function of order $n\geq 1$ on some axially symmetric slice domain. Then, we have

$$\displaystyle \mathcal{D}^{n-1}\mathcal{C}_n(f)(q)=\frac{1}{2^{n-1}}\tau_n(f)(q), \forall q\in \Omega.$$

\end{thm}
\begin{proof}
 Since $\Omega$ is a slice domain, by the poly-decomposition for slice polyanalytic functions there exist $f_0,..,f_{n-1}\in \mathcal{SR}(\Omega)$ such that  $$f(q)=\displaystyle \sum_{k=0}^{n-1}\overline{q}^kf_k(q), \forall q\in \Omega.$$
Thus, by Proposition \ref{GVAct} gives  $$\displaystyle V(f)(q)=\sum_{k=1}^{n-1} 2k\overline{q}^{k-1}f_k(q), \forall q\in \Omega.$$
In a similar way, we apply $(n-1)$ times the global operator $V$  and use Proposition \ref{GVAct} to get
$$
\displaystyle V^{n-1}(f)(q)=2^{n-1}(n-1)!f_{n-1}(q), \textbf{  } \forall q\in \Omega .$$
As a direct consequence, by definition of $\tau_n$ we have

\begin{equation}\label{taunf}
\displaystyle \tau_n (f)(q)=2^{n-1}(n-1)!\Delta f_{n-1}(q), \textbf{  } \forall q\in \Omega.
\end{equation}
On the other hand, since $(f_k)_{0 \leq k\leq n-1}$ are all slice hyperholomorphic we know by the Fueter mapping theorem that $$\mathcal{D}(\Delta f_k)=0, \textbf{ }\forall 0 \leq k\leq n-1.$$
Therefore, by Leibniz rule for the Cauchy-Fueter operator we have

\begin{equation}\label{Dact}
\displaystyle \mathcal{D}(x_0^k\Delta f_k)(q)=kx_{0}^{k-1}\Delta f_k(q); \textbf{  } \forall q\in \Omega,  \forall 0 \leq k\leq n-1.
\end{equation}
We know by definition of $\mathcal{C}_n$ that
$$\displaystyle \mathcal{C}_n(f)(q)=\sum_{k=0}^{n-1}x_0^k\Delta(f_k)(q), \forall q\in \Omega.$$
Thus, we use \eqref{Dact} and get

$$\displaystyle \mathcal{D}[\mathcal{C}_n(f)](q)=\sum_{k=1}^{n-1}kx_{0}^{k-1}\Delta f_k(q), \forall q\in \Omega.$$
Similarly, if we apply the Cauchy-Fueter operator $(n-1)$ times and use \eqref{Dact}, with some computations we get
\begin{equation}\label{DCn}
\displaystyle \mathcal{D}^{n-1} [\mathcal{C}_n(f)](q)=(n-1)!\Delta f_{n-1}(q), \forall q\in \Omega.
\end{equation}
Finally, we combine the relations \eqref{taunf} and \eqref{DCn} to conclude that $$\displaystyle \mathcal{D}^{n-1}\mathcal{C}_n(f)(q)=\frac{1}{2^{n-1}}\tau_n(f)(q), \forall q\in \Omega.$$
\end{proof}
\section{The Poly-Fueter mapping theorem in integral form}
In this section,  we prove a  Cauchy integral theorem and Cauchy formula for slice polyanaytic functions. Then, we study some integral representation of the poly-Fueter mapping theorem on the quaternionic unit ball that will extend the results obtained in \cite{ColomboSabadiniSommen2010}.\\
We recall the polyanalytic Cauchy formula in complex analysis, see Theorem 2.1 in \cite{DB1978}.
\begin{thm}
For $k \geq 1$, we set $$\psi_k(z)=\displaystyle \frac{1}{2\pi i} \frac{\bar{z}}{|z|^2}\frac{Re(z)^{k-1}}{(k-1)!}.$$ For $z=x+iy$, set $d \sigma=dx \wedge dy$. If $f$ is polyanalytic of order $n$, then for all $z\in\mathbb{D}$ we have  $$f(z)=\displaystyle \int_{\partial \mathbb{D}}\sum _{j=0}^{n-1}(-2)^j\psi_{j+1}(u-z)\frac{\partial^j}{\partial \bar{u}^j}f(u)d\sigma. $$
\end{thm}
First, we prove a version of the Cauchy's integral theorem for slice polyanalytic functions
\begin{thm}
Let $f$ and $g$ be a left and right slice polyanalytic functions of order $n$ respectively on some axially symmetric slice domain $\Omega$ containing the closure of $\mathbb{B}$. Then, for any $I\in\Sq$ we have
$$\displaystyle\int_{\partial\mathbb{B}_I}\sum_{j=0}^{n-1}(-1)^jg\overline{\partial_{I}}^{n-1-j}dw_I \overline{\partial_I}^j f=0, \textbf{   } $$
where $dw_I=-dw I$for $w\in\C_I.$
\end{thm}
\begin{proof}
Let $I\in\Sq$ and choose $J\in\Sq$ be such that $I\perp J$. Thus, by Splitting Lemma for slice polyanalytic functions proved in \cite{ADS2019} we can write $$f(w)=F_1(w)+F_2(w)J \text{ and } g(w)=G_1(w)+JG_2(w),$$  where $F_l,G_l:\mathbb{B}_I\longrightarrow \C_I$ for $l=1,2$ are complex polyanalytic functions of order $n$. In order to simplify the computations, we set $$\Phi(f,g):=\displaystyle\int_{\partial\mathbb{B}_I}\sum_{j=0}^{n-1}(-1)^jg\overline{\partial_{I}}^{n-1-j}dw_I \overline{\partial_I}^j f$$

Then, direct computations lead to
$$\Phi(f,g)=\Phi(F_1,G_1)+\Phi(F_2,G_1)J+J\Phi(F_1,G_2)+J\Phi(F_2,G_2)J$$
At this stage, we apply the poly-Cauchy integral theorem proved in \cite{DB1978} to deduce that $$\Phi(F_1,G_1)=\Phi(F_2,G_1)=\Phi(F_1,G_2)=\Phi(F_2,G_2)=0.$$
This ends the proof.
\end{proof}
Now, let $n\geq 1$ and $w\in\mathbb{B}$ be such that $w\in\C_J$ with $J\in\mathbb{S}$. For all $0 \leq j\leq n-1$, we consider the function defined by $$\phi_{j,w}(z)=\displaystyle \frac{1}{w-z}\frac{(Re(w-z))^j}{j!}; \textbf{ }z\in \mathbb{B}_J, z\neq w.$$
Then, we have
\begin{prop}\label{PCK}
For all $0 \leq j \leq n-1$, the slice polyanalytic extension of $\phi_{j,w}$ is given by
$$\phi_{j,w}(q)=S^{-1}(w,q)\frac{(Re(w-q))^j}{j!} \textbf{ } \forall q\in\mathbb{B}, q\notin [w],$$
where $S^{-1}(w,q)$ is the slice hyperholomorphic Cauchy kernel.
\end{prop}
\begin{proof}
Let $0 \leq j \leq n-1$. We know that $S^{-1}(w,q)$ is left slice regular with respect to the variable $q$. Moreover, it is clear that $\displaystyle q\mapsto \frac{(Re(w-q))^j}{j!} $ is a real valued slice polyanalytic function of order $n$ for all $ 0 \leq j\leq n-1$. So, we can apply Proposition 3.3 in \cite{ADS2019} to see that the product $\displaystyle S^{-1}(w,q)\frac{(Re(w-q))^j}{j!}$ is slice polyanalytic of order $n$ with respect to the variable $q$. And since it coincides with $\phi_{j,w}(z)$ on $\mathbb{B}_J$ the proof ends thanks to the identity principle (see \cite{ADS2019}).
\end{proof}
\begin{rem}
Another way to prove Proposition \ref{PCK} consists of using the extension Lemma for slice polyanalytic functions, see \cite{ADS2019}. Indeed, we note that $z\mapsto \phi_{j,w}(z)$ is polyanalytic of order $n$ for any $z\neq w$. Thus, it admits a unique slice polyanalytic extension denoted by $ext[\phi_{j,w}(z)](q)$. By definition, for $q=x+I_qy$ and $z=x+Jy$ such that $q \notin [w]$ we have
\[ \begin{split}
 \displaystyle ext[\phi_{j,w}(z)](q) &= \frac{1}{2}[\phi_{j,w}(z)+\phi_{j,w}(\overline{z})]+\frac{I_qJ}{2}[\phi_{j,w}(\overline{z})-\phi_{j,w}(z)] \\
 &=ext\left(\frac{1}{w-z}\right)\frac{(Re(w-q))^j}{j!}  \\
&=S^{-1}(w,q)\frac{(Re(w-q))^j}{j!}, \\
\end{split}
\]
where $S^{-1}(w,q)$ is the slice hyperholomorphic Cauchy kernel given by $$\displaystyle S^{-1}(w,q)=(w-\overline{q})(w^2-2Re(q)w+|q|^2)^{-1}.$$
\end{rem}
\begin{prop}
Let $q,w\in\mathbb{B}$ be such that $q\notin[w]$. The function, $\phi_{j,w}(q)$ is right slice polynalytic of order $j+1$ in the variable $w$.
\end{prop}
\begin{proof}
The proof is easy using the fact that $S^{-1}(w,q)$ is right slice regular in $w$ combined with the right version of Proposition 3.3 in \cite{ADS2019}.
\end{proof}
\begin{thm}\label{PCF}
Let $\Omega$ be an axially symmetric slice domain containing the closure of $\mathbb{B}$ and  $f:\Omega\longrightarrow \Hq$ a slice polyanalytic function of order $n\geq 1$. For $I\in\mathbb{S}$, set $dw_I=-dw I$.
The integral below $$\displaystyle \frac{1}{2\pi}\int_{\partial \mathbb{B}_I}\sum_{j=0}^{n-1}(-2)^jS^{-1}(w,q)\frac{(Re(w-q))^j}{j!}dw_I\overline{\partial_I}^j(f)(w),$$ does not depend on the choice of the imaginary unit $I\in\mathbb{S}$.

Moreover, for all $q\in\mathbb{B}$ we have the integral representation
$$f(q)=\displaystyle \frac{1}{2\pi}\int_{\partial \mathbb{B}_I}\sum_{j=0}^{n-1}(-2)^jS^{-1}(w,q)\frac{(Re(w-q))^j}{j!}dw_I\overline{\partial_I}^j(f)(w). $$

\end{thm}
\begin{proof}
The independence of the choice of $I\in\Sq$ is a direct consequence of the poly-decomposition in Proposition \ref{carac3} combined with the series expansion theorem for slice hyperholomorphic functions. To show the second part of the statement,  let $J\in\mathbb{S}$ be such that $J\perp I$. We know that $f\in\mathcal{SP}_n(\mathbb{B})$, so by Proposition 3.4 in \cite{ADS2019} there exist two polyanalytic functions $F,G:\mathbb{B}_J\longrightarrow \C_J$ of order $n$ such that for any $w\in\mathbb{B}_J$ we have $$f(w)=F(w)+G(w)J.$$
In particular, $$\overline{\partial_I}^j f(w)=\overline{\partial_I}^jF(w)+\overline{\partial_I}^jG(w)J.$$

Then, we have on $\mathbb{B}_I$ the following reproducing property thanks to the complex poly-Cauchy formula applied to $F$ and $G$
\[ \begin{split}
 \displaystyle \frac{1}{2\pi}\int_{\partial \mathbb{B}_I}\sum_{j=0}^{n-1}(-2)^jS^{-1}(w,q)\frac{(Re(w-q))^j}{j!}dw_I\overline{\partial_I}^j(f)(w)
 &=F(q)+G(q)J  \\
&=f(q). \\
\end{split}
\]
Furthermore, in Proposition \ref{PCK} we deal with a slice polyanalytic kernel. So, the function $$\displaystyle \Psi(q)=\int_{\partial \mathbb{B}_I}\sum_{j=0}^{n-1}(-2)^jS^{-1}(w,q)\frac{(Re(w-q))^j}{j!}dw_I\overline{\partial_I}^j(f)(w),$$
is also slice polyanayltic of order $n$. Hence, we can conclude by Identity principle since $\Psi$ coincides with $f$ on $\mathbb{B}_I$.

\end{proof}
\begin{rem}
The case $n=1$ in the previous theorem gives the slice hyperholomorphic Cauchy formula that can be found in \cite{CSS}.
\end{rem}
As a direct application of the slice poly Cauchy formula we will prove the poly-Fueter mapping theorem in its integral form. To this end, we need some technical lemmas. First, for every $n\geq 1$,  $1 \leq j\leq n-1$ and $w\in\partial \mathbb{B}$, denote by $\mathcal{F}_j(w,q)$ the quaternionic valued function on $\mathbb{B}$ sending $q$ into
\begin{equation}
\mathcal{F}_j(w,q):=S^{-1}(w,q)\frac{Re^j(w-q)}{j!},
\end{equation}
where $Re^j(w-q):=\left(Re(w-q)\right)^j$.
\begin{lem}\label{VFj}
Let $w\in\partial\mathbb{B}$. Then, for every $q\in\mathbb{B}$,  we have
$$\displaystyle V(\mathcal{F}_0(w,q))=0 \text{ and }V(\mathcal{F}_j(w,q))=-\mathcal{F}_{j-1}(w,q), \forall j\geq 1.$$
\end{lem}
\begin{proof}
First, we have $\mathcal{F}_0(w,q)=S^{-1}(w,q)$ is the slice hyperholomorphic Cauchy kernel. So, $q\longmapsto \mathcal{F}_0(w,q)$ is slice hyperholomorphic with respect to the variable $q$. Thus, we have $V(\mathcal{F}_0(w,q))=0 $ for all $q$. On the other hand, for all $j\geq 1$ we have
$$\displaystyle G(\mathcal{F}_j(w,q))=G\left(S^{-1}(w,q)\frac{Re^j(w-q)}{j!}\right),\textbf{ } \forall q\in \mathbb{B}.$$
Then, we apply Proposition \ref{GPRO} on which we see how the global operator $G$ acts on the product keeping in mind that one of the functions is real valued and get
\begin{equation}\label{GFj}
\displaystyle G(\mathcal{F}_j(w,q))=S^{-1}(w,q)G\left(\frac{Re^j(w-q)}{j!}\right),\textbf{ } \forall q\in \mathbb{B}.
\end{equation}

However, we have
\[ \begin{split}
 \displaystyle G\left(\frac{Re^j(w-q)}{j!}\right) &=|\vec{q}\,|^2\partial_{x_0}\left(\frac{Re^j(w-q)}{j!}\right)  \\
 &=-|\vec{q}\,|^2\frac{Re^{j-1}(w-q)}{(j-1)!}
\\
\end{split}
\]
Then, we replace in \eqref{GFj} and get
$$\displaystyle G(\mathcal{F}_j(w,q))=-|\vec{q}\,|^2S^{-1}(w,q)\frac{Re^{j-1}(w-q)}{(j-1)!},\textbf{ } \forall q\in \mathbb{B}.$$
Hence, we use Remark \ref{GVR} to see that the result holds outside the real line. Then, we apply again Theorem 2.4 in\cite{P2019} which allows to extend the formula everywhere on $\mathbb{B}$.  Finally, we conclude that for any $q\in\mathbb{B}$ we have
$$V(\mathcal{F}_j(w,q))=-\mathcal{F}_{j-1}(w,q), \forall j\geq 1.$$
This ends the proof.
\end{proof}
\begin{lem}
Let $w\in\partial\mathbb{B}$. For any $n\geq 1$, we set $$\tau_n=\Delta\circ V^{n-1}.$$ Then, for every $q\in\mathbb{B}$, we have
\begin{enumerate}
\item $\tau_1(\mathcal{F}_0(w,q))=\Delta S^{-1}(w,q)$.
\item For all $n\geq 2$, we have
\begin{enumerate}
\item $\tau_n(\mathcal{F}_j(w,q))=0, \forall 0 \leq j < n-1$.
\item $\tau_n(\mathcal{F}_{n-1}(w,q))=(-1)^{n-1}\Delta S^{-1}(w,q).$
\end{enumerate}
\end{enumerate}
\end{lem}
\begin{proof}
(1) It is immediate by the definition of the map $\tau_1=\Delta$. \\
(2) We reason by induction. First, we note that for $n=2$, $\mathcal{F}_0(w,q)$ is slice hyperholomorphic with respect to $q$ so that $$\tau_2(\mathcal{F}_0(w,q)=\Delta\circ V(\mathcal{F}_0(w,q))=0.$$
Moreover, we have $$\displaystyle \tau_2(\mathcal{F}_1(w,q))=\Delta \left( V(\mathcal{F}_1(w,q))\right).$$
Moreover,  Lemma \ref{VFj} yields $$\displaystyle V(\mathcal{F}_1(w,q))=-\mathcal{F}_0(w,q)$$
so we get $$\displaystyle \tau_2(\mathcal{F}_1(w,q))=-\Delta (\mathcal{F}_0(w,q))=-\Delta S^{-1}(w,q).$$
 We conclude that the result holds for $n=2$. Let us suppose by induction that the assertions (a), (b) in the statement hold for $n\geq 2$ and we prove them for $n+1$.

(a) Let $w\in\partial\mathbb{B}$. Then, for every $q\in\mathbb{B}$, it is clear that $$\tau_{n+1}(\mathcal{F}_0(w,q))=\Delta\circ V^{n}(\mathcal{F}_0(w,q))=0.$$
We observe that
\begin{equation}\label{n+1}
\tau_{n+1}=\Delta \circ V^{n}=\Delta\circ V^{n-1}\circ V=\tau_n\circ V.
\end{equation}
Then, for all $ 1 \leq j <n$ making use of Lemma \ref{VFj} we have
\[ \begin{split}
 \displaystyle \tau_{n+1}(\mathcal{F}_j(w,q)) &=\tau_n\circ V(\mathcal{F}_j(w,q))   \\
 &=-\tau_n(\mathcal{F}_{j-1}(w,q)) \\
 &=-\tau_n(\mathcal{F}_h(w,q)); \textbf{ } 0 \leq h=j-1< n-1. \\
\end{split}
\]
Therefore, by induction hypothesis we conclude that $$\displaystyle \tau_{n+1}(\mathcal{F}_j(w,q))=0, \forall 0 \leq j < n.$$
This shows that (a) holds.

(b) We use a second time the observation \eqref{n+1} combined with Lemma \ref{VFj} and get by induction hypthesis
\[ \begin{split}
 \displaystyle \tau_{n+1}(\mathcal{F}_n(w,q)) &=\tau_n\circ V(\mathcal{F}_n(w,q))   \\
 &=-\tau_n(\mathcal{F}_{n-1}(w,q)) \\
 &=(-1)^n\Delta S^{-1}(w,q). \\
\end{split}
\]
Hence, (b) also holds. This ends the proof.
\end{proof}

\begin{thm}\label{PFI}
Let $f$ be a slice polyanalytic function of order $n\geq 1$ on some axially symmetric slice domain $\Omega$ that contains the closure of $\mathbb B$. Then,  the Fueter regular function $\tau_n(f)$ given by $$\tau_n(f)(q)=\Delta\circ V^{n-1}(f)(q)$$
has the integral representation
$$\displaystyle \tau_n(f)(q)=c(n,\pi)\int _{\partial \mathbb{B}_I} \Delta S^{-1}(w,q)dw_I\overline{\partial}_{I}^{n-1}(f)(w), \forall q\in \mathbb{B}$$
where $I\in\mathbb{S}$ and $\displaystyle c(n,\pi)=\frac{2^{n-1}}{2\pi}.$
\end{thm}
\begin{proof}
Let $f\in\mathcal{SP}_n(\Omega)$, we know by the poly-Cauchy formula for slice polyanalytic functions (Theorem \ref{PCF}) that for all $q\in\mathbb{B}$ we have
$$\displaystyle f(q)=\frac{1}{2\pi}\int_{\partial \mathbb{B}_I}\sum_{j=0}^{n-1}(-2)^j\mathcal{F}_j(w,q)dw_I\overline{\partial}^j(f)(w).$$
Therefore, we apply the Fueter mapping $\tau_n=\Delta\circ V^{n-1}$ and obtain that
$$\displaystyle \tau_n(f)(q)=\frac{1}{2\pi}\int_{\partial \mathbb{B}_I}\sum_{j=0}^{n-1}(-2)^j\tau_n(\mathcal{F}_j(w,q))dw_I\overline{\partial}^j(f)(w), \textbf{  } \forall q\in \mathbb{B}\setminus \R.$$
However, by Lemma \ref{VFj} we know that
$$\tau_n(\mathcal{F}_{n-1}(w,q))=(-1)^{n-1}\Delta S^{-1}(w,q) \text{ and } \tau_n(\mathcal{F}_j(w,q))=0, \forall 0 \leq j < n-1$$

Hence, we obtain $$\displaystyle \tau_n(f)(q)=\frac{2^{n-1}}{2\pi}\int _{\partial \mathbb{B}_I} \Delta S^{-1}(w,q)dw_I\overline{\partial}_{I}^{n-1}(f)(w), \forall q\in \mathbb{B}\setminus \R.$$
Finally, it is clear that the integral in the right hand side is Fueter regular with respect to $q$ everywhere on $\mathbb{B}$ which allows to extend $\tau_n(f)$ to a Fueter regular function on the unit ball. This completes the proof.
\end{proof}

\begin{cor}
Under the same hypothesis of Theorem \ref{PFI} we note that the poly-Fueter mapping has the explicit integral expression
$$\displaystyle \tau_n(f)(q)=\frac{2^n}{\pi}\int _{\partial \mathbb{B}_I}(\overline{q}-w)(w^2-2Re(q)w+|q|^2)^{-2}dw_I\overline{\partial}_{I}^{n-1}(f)(w), \forall q\in \mathbb{B}.$$
\end{cor}
\begin{proof}
We apply Theorem \ref{PFI} and use the expression  $$\Delta S^{-1}(w,q)=-4(w-\overline{q})(w^2-2Re(q)w+|q|^2)^{-2},$$
that was proved in \cite{ColomboSabadiniSommen2010}.
\end{proof}
\begin{rem}

\begin{enumerate}
\item Thanks to Theorem \ref{TauC} the integral formulation of the poly-Fueter mapping theorem can be expressed in terms of the map $\mathcal{C}_n$ as
$$\displaystyle \mathcal{D}^{n-1}[\mathcal{C}_n(f)](q)=\frac{1}{2\pi}\int _{\partial \mathbb{B}_I} \Delta S^{-1}(w,q)dw_I\overline{\partial}_{I}^{n-1}(f)(w), \forall q\in \mathbb{B}.$$
\item The case $n=1$ in Theorem \ref{PFI} is the Fueter mapping theorem in integral form proved in \cite{ColomboSabadiniSommen2010}.
\end{enumerate}

\end{rem}

\section{The slice polymonogenic case}
In this section, we see how the results of quaternionic slice polyanalytic functions can be reformulated in the slice monogenic setting. We omit to write the proofs since they are similar to the quaternionic case. We recall first some basic notations, let $\lbrace{e_1,e_2, . . . , e_n}\rbrace$ be an orthonormal basis of the Euclidean vector space $\R^n$ satisfying the rule
$$e_ke_s + e_se_k = -2\delta_{k,s}, k, s= 1, . . . , n$$
where $\delta_{k,s}$ is the Kronecker symbol. The set $$\lbrace{e_A : A \subset\lbrace{1, . . . , n}\rbrace \text{ with }
e_A = e_{h_1}e_{h_2}...e_{h_r}, 1 \leq h_1 < ...< h_r \leq n, e_{\emptyset} = 1}\rbrace $$
forms a basis of the $2^n$-dimensional Clifford algebra $\R_n$ over $\R$. Let $\R^{n+1}$ be embedded in $\R_n$ by identifying
$(x_0, x_1,..., x_n) \in \R^{n+1}$ with the paravector $x=x_0+\underline{x}\in \R_n$. The conjugate of $x$ is given by $\bar{x} = x_0-\underline{x}$
and the norm $\vert{x}\vert$ of $x$ is defined by $\vert{x}\vert^2=x_0^2+...+x_n^2$. We denote also by $\mathbb{S}^{n-1}$ the $(n-1)$-dimensional sphere of unit vectors in $\R^n$ given by $$\mathbb{S}^{n-1}=\lbrace{\omega=x_1e_1+...+x_ne_n\textbf{}: x_1^2+...+x_n^2=1}\rbrace, \textbf{  } \omega^2=-1.$$
The Euclidean Dirac operator on $\R^n$ is given by $$\displaystyle D_{\underline{x}}=\sum_{j=1}^ne_j\partial_{x_j}.$$  The generalized Cauchy-Riemann operator (also known as Weyl operator) and its conjugate in $\R^{n+1}$ are given respectively by $$\displaystyle D:=\partial_{x_0}+D_{\underline{x}} \text{ and }  \displaystyle\overline{D}:=\partial_{x_0}-D_{\underline{x}}.$$  Real differentiable functions on some open subset of $\R^{n+1}$ taking their values in $\R_n$ that are in the kernel of $D^k$ are called left $k$-monogenic or polymonogenic of order $k$, see \cite{DB1978}. We consider also the slice monogenic version given by
\begin{defn}
Let $U$ be an axially symmetric open set in $\R^{n+1}$ and $f:U\longrightarrow \R_n$ be a slice function of class  $\mathcal{C}^k$. We say that $f$ is slice polymonogenic of order $k$ or s-polymonogenic for short, if for any $I\in\Sq^{n-1}$, we have
$$\overline{\partial_I}^kf_I(x+Iy)=0.$$
The set of slice polymonogenic functions of order $k$ is denoted $\mathcal{SM}_k(U)$.
\end{defn}
\begin{rem}
\begin{enumerate}
\item The set $\mathcal{SM}_k(U)$ forms a right module on $\R_n$.
\item The case $k=1$ corresponds to the slice monogenic functions considered in \cite{CSS}.
\end{enumerate}
\end{rem}
We consider on $\Omega \subset \R^{n+1}$ the global operator on high dimensions defined by $$\displaystyle V_n(f)(x):=\partial_{x_0}f(x)+\frac{\vec{x}}{|\vec{x}|^2}\sum_{l=1}^nx_l \partial_{x_l}f(x), \forall x\in\Omega\setminus \R.$$
\begin{lem}[Splitting Lemma]
Let $U$ be an axially symmetric open set in $\R^{n+1}$ and $f:U\longrightarrow \R_n$ be a slice polymonogenic function of order $k$. For every $I=I_1\in\Sq$ let $I_2,...,I_n$ be a completion to an orthonormal basis of $\R_n$. Then, there exists $2^{n-1}$ polyanayltic functions of order $k$ denoted $F_A:U_I\longrightarrow \C_I$ such that for every $z=x+Iy$
$$f_I(z)=\displaystyle \sum_{|A|=0}^{n-1}F_A(z)I_A, \textbf{  } I_A=I_{i_1}...I_{i_l},$$
where $A=\lbrace i_1,...,i_l \rbrace$ is a subset of $\lbrace 2,...,n \rbrace$, with $i_1<...<i_l$.
\end{lem}
\begin{thm}[s-polymonogenic decomposition]\label{spd}
Let $\Omega$ be an axially symmetric slice domain of $\R^{n+1}$ and $f:\Omega\longrightarrow \R_n$. Then, $f\in\mathcal{SM}_k(\Omega)$ if and only if there exists unique $ f_0,...,f_{k-1}\in\mathcal{SM}(\Omega)$ such that $$f(x)=f_0(x)+\overline{x}f_1(x)+...+\overline{x}^{k-1}f_{k-1}(x), \textbf{  } \forall x\in \Omega.$$
\end{thm}
Using similar calculations to the quaternions case, we can prove that
\begin{thm}
Let $\Omega$ be an axially symmetric slice domain of $\R^{n+1}$ and $f:\Omega\longrightarrow \R_n$ an s-polymonogenic function of order $k\geq 1$. Then, $f$ belongs to $\ker(V^k_n)$, i.e:
$$\displaystyle V^k_n(f)(x)=0, \textit{ }\forall x\in\Omega.$$
\end{thm}
For slice polymonogenic functions we state the poly-Sce-Fueter mapping theorems in the Clifford setting as follows
\begin{thm}[Poly-Fueter-Sce mapping theorem I]
Let $n$ be an odd number and $\Omega$ an axially symmetric slice domain of $\R^{n+1}$. If $f$ is an s-polymonogenic function of order $k$. Then, the poly-Fueter mapping defined by $$\tau_{n,k}(f)(x)=\displaystyle \Delta_{\R^{n+1}}^{\frac{n-1}{2}}V_n^{k-1}f(x)$$ is a polymonogenic function of order $k$.
\end{thm}
\begin{thm}[Poly-Fueter-Sce mapping theorem II]
Let $\Omega$ be an axially symmetric slice domain of $\R^{n+1}$ and $f:\Omega\longrightarrow \R_{n}$ a slice polyanalytic function of order $k\geq 1$. Assume that $f$ admits a poly-decomposition given by

$$\displaystyle f(x)=\sum_{j=0}^{k-1}\overline{x}^jf_j(x), \forall x\in\Omega$$ where $f_0,...,f_{n-1}\in \mathcal{SM}(\Omega).$
Then, the function defined by
\begin{equation}
\displaystyle \mathcal{C}_{n,k}(f)(q)=\sum_{j=0}^{k-1}x_0^j\Delta_{\R^{n+1}}^{\frac{n-1}{2}}(f_j)(x), \forall x\in \Omega
\end{equation}

is a poly-monogenic function of order $k$.
\end{thm}
\noindent{\bf Acknowledgements} \\
The authors would like to thank the referees for the useful comments that improve the presentation of the paper.\\
Daniel Alpay thanks the Foster G. and Mary McGaw Professorship in
Mathematical Sciences, which supported this research. Kamal Diki thanks Chapman university for kind hospitality during the period in which a part of this paper was written. His research is supported by the project INdAM Doctoral Programme in Mathematics and/or Applications Cofunded by Marie Sklodowska-Curie Actions, acronym: INdAM-DP-COFUND-2015, grant number: 713485.

\end{document}